\documentclass[11pt]{amsart}
\usepackage{amsmath,amsfonts,amsthm,amscd,amssymb,mathrsfs,amssymb}
\usepackage{mathrsfs}
\newtheorem{theorem}{Theorem}

\newtheorem{proposition}[theorem]{Proposition}
\newtheorem{lemma}[theorem]{Lemma}

\newtheorem{remark}[theorem]{Remark}
\newcommand{\CP}{\mathbb{CP}}

\newcommand{\RR}{\mathbb{R}}
\newcommand{\ZZ}{\mathbb{Z}}

\newcommand{\U}{{\rm{U}}}
\newcommand{\LB}{{\rm{LB}}}
\renewcommand{\H}{\mathcal{H}}
\newcommand{\w}{\omega}

\renewcommand{\i}{i}

\newcommand{\set}{\,|\,}
\newcommand{\proofend}{\hfill$\square$}

\numberwithin{equation}{section}
\numberwithin{theorem}{section}
\begin{document}
\bibliographystyle{alpha} 
\title[Toric LeBrun metrics]{Toric LeBrun metrics
and Joyce metrics
}
\author{Nobuhiro Honda}
\address{Mathematical Institute, Tohoku University,
Sendai, Miyagi, Japan}
\email{honda@math.tohoku.ac.jp}\author{Jeff Viaclovsky}
\address{Department of Mathematics, University of Wisconsin, Madison, 
WI, 53706}
\email{jeffv@math.wisc.edu}
\thanks{The first author has been partially supported by the Grant-in-Aid for Young Scientists  (B), The Ministry of Education, Culture, Sports, Science and Technology, Japan. The second author 
has been partially supported by the National Science Foundation under 
grant DMS-1105187.
}
\begin{abstract}
We show that, on the connected sum of complex projective planes, 
any toric LeBrun metric can be identified with 
a Joyce metric
admitting a semi-free circle action through
an explicit conformal equivalence.
A crucial ingredient of the proof is an 
explicit connection form for toric LeBrun metrics. 
\end{abstract}
\date{August 10, 2012}
\maketitle

\section{Introduction}

  The subject of self-dual metrics on four-manifolds 
has rapidly developed since the discovery by Poon of
a $1$-parameter family of self-dual conformal classes on 
$\CP^2 \# \CP^2$ \cite{Poon1986}. We do not attempt to give a complete 
review of subsequent developments here; in this short note we are 
concerned only with two classes of self-dual metrics
on $n \# \CP^2$. 

First, in 1991, Claude LeBrun \cite{LeBrun1991} produced explicit examples 
with ${\rm{U}}(1)$-symmetry on $n \# \CP^2$, using a hyperbolic ansatz 
inspired by the Gibbons-Hawking ansatz \cite{GibbonsHawking}.   
LeBrun's construction depends on the choice of $n$ points 
in hyperbolic 3-space $\mathcal{H}^3$. For $n = 2$, the only 
invariant of the configuration is the distance 
between the monopole points, and LeBrun conformal classes 
are the same as the $1$-parameter family found by Poon. 
The Poon metrics are toric, that is, they admit a 
smooth effective action by a real torus ${\rm{U}}(1) \times {\rm{U}}(1)$.
For $n > 2$, a LeBrun metric admits a torus action if and 
only if the monopole points belong to a common hyperbolic 
geodesic. These form a sub-class of LeBrun metrics, 
which we call {\em{toric LeBrun metrics}}. 

The second class of metrics we are concerned with are the 
metrics on $n \# \CP^2$ discovered by Dominic Joyce in 
\cite{Joyce1995}. Joyce's construction depends on the choice of 
$n + 2$ points on the {\em{boundary}} of hyperbolic $2$-space.
These metrics are always toric. It was subsequently 
shown by Fujiki that any compact toric self-dual four-manifold 
with non-zero Euler characteristic is necessarily diffeomorphic
to $n \# \CP^2$, and furthermore
the self-dual structure is of Joyce-type \cite{Fujiki2000}.

We recall that a {\em{semi-free}} action is a non-trivial action 
of a group $G$ on a connected space $M$ such that for every $x \in M$, the 
corresponding isotropy subgroup is either all of $G$ or is trivial. 
Many of the families of metrics constructed by Joyce are not 
of LeBrun-type. However,  if the torus action contains a circle subgroup
which acts semi-freely on $n \# \CP^2$,  they 
are the same (for each $n$, such a torus action is unique).
This coincidence was stated in \cite{Joyce1995}, but without 
proof. This fact follows from Fujiki's theorem mentioned 
above, however we feel it is useful to have a direct proof. 
Recently, the authors determined the conformal automorphism 
groups of LeBrun's monopole metrics \cite{HondaViaclovsky}. In the 
course of that work an explicit connection for any toric
LeBrun metric was found, which we use in this paper 
to prove the following:
\begin{theorem}
\label{L=J}
On $n \# \CP^2$, the class of toric LeBrun metrics and the class of 
Joyce metrics admitting a semi-free circle action are the same,
and any metric of the first class can be identified 
with a metric of the second class through an explicit 
conformal equivalence. 
\end{theorem}


\section{An explicit global connection}
\label{explicit}
First we quickly recall the construction of LeBrun's self-dual hyperbolic monopole metrics.
Let $\mathcal H^3=\{(x,y,z)\set z>0\}$ be equipped with the usual hyperbolic metric
$g_{\mathcal H^3}:=(dx^2+dy^2+dz^2)/z^2$.
Let $n$ be any non-negative integer 
and  $P = \{p_1, \dots,  p_n\}$ be distinct points in $\H^3$.
Let $\Gamma_{p_{\alpha}}$ be the fundamental solution for the hyperbolic 
Laplacian based at $p_{\alpha}$ with normalization 
$\Delta \Gamma_{p_{\alpha}} = -2 \pi \delta_{p_{\alpha}}$,
and define
\begin{align}\label{V1}
V = 1 + \sum_{{\alpha} = 1}^n  \Gamma_{p_{\alpha}}.
\end{align}
Then $* dV$ is a closed $2$-form on $\H^3 \setminus P$,
and $[* dV]/2\pi$ belongs to an integral class  $
H^2 ( \H^3 \setminus P, \ZZ )$.
Let $\pi: X_0 \rightarrow \H^3 \setminus P$ 
be the unique principal $\U(1)$-bundle determined by this 
 integral class.
By Chern-Weil theory, there is a connection form $\w \in H^1(X_0, \i \RR)$
such that $d\w=\i (* dV)$. Then LeBrun's metric is defined by 
\begin{align}
\label{LBmetric1}
g_{\LB} = z^2 (  V \cdot g_{\H^3} - \frac1V \, \w \odot \w).
\end{align}
This is anti-self-dual with respect to a K\"ahler
orientation of $X_0$.
By attaching points $\tilde p_{\alpha}$ over each $p_{\alpha}$,
we obtain a complete, K\"ahler scalar-flat 
(and therefore anti-self-dual) ALE manifold,
which can be conformally compactified by adding a point 
at infinity, yielding a self-dual conformal class on $n\#\CP^2$.
The $\U(1)$-action of the principal $\U(1)$-bundle 
naturally extends to $n\#\CP^2$,
and the resulting $\U(1)$-action on $n\#\CP^2$ is semi-free.
See \cite{LeBrun1991,LeBrun1993} for detail.

From the construction LeBrun metrics always admit a
$\U(1)$-action, and they admit an effective $\U(1) \times  \U(1)$-action if and only if 
all the $n$ points belong to a common geodesic.
We call these latter metrics {\em{toric LeBrun metrics}}.
By applying a hyperbolic isometry, without loss of generality we may assume 
that the geodesic is the $z$-axis,
and we let $p_{\alpha}=(0,0,c_{\alpha})$ with 
 $0<c_1<c_2<\cdots<c_n$.
We also define $c_0=0$ and $c_{n+1} = \infty$.

For toric LeBrun metrics, we shall explicitly write down the
connection form $\omega$  on the $\U(1)$-bundle $\pi:X_0\to \mathcal H^3$.
For this, we first let $U = \H^3 \setminus \{z\mbox{-axis}\}$ and 
take cylindrical coordinates on $U$ as
\begin{align}
U = \{ (x,y,z) = ( r \cos \tau, r \sin \tau, z ) \set z > 0, \ 0 \leq \tau < 2 \pi \}
\end{align}
(we use $\tau$ for the angular coordinate here, since $\theta$ will 
be used below as the angular coordinate on the circle bundle).
Also on the $z$-axis we define an interval 
\begin{align}
I_{\alpha}:=\{(0,0,z)\set c_{\alpha-1}<z<c_{\alpha}\}, \ \ \ 1\le {\alpha}\le n+1,
\end{align}
and we let $U_{\alpha}:= U\cup I_{\alpha}$ for each $1\le \alpha\le n+1$.
Then we obtain an open covering 
\begin{align}\label{covering}
\H^3 \setminus \{p_1, p_2, \ldots, p_n \} = U_1 \cup U_2 \cup \ldots
\cup U_{n+1}.
\end{align}
Finally, for any positive real number $c$, we define a function $f_c$ by
\begin{align}
f_c(r,z) = \frac{r^2 + z^2 - c^2}{2 \sqrt{ (c^2 + r^2 + z^2)^2 - 4 c^2 z^2}} - \frac 12.
\end{align}
We note that $(c^2 + r^2 + z^2)^2 - 4 c^2 z^2 \geq 0$ and is zero only at $(0,0,c)$.
Therefore, $f_{c}$ is a function defined on  all of $\H^3 \setminus \{(0,0,c)\}$.

\begin{theorem}
\label{conn}
Using the above notation, 
define a function on $\H^3 \setminus \{p_1, p_2, \ldots, p_n \}$ by
 $f : = f_{c_1} + f_{c_2} + \cdots + f_{c_n} $. Then $f$ satisfies
\begin{align}
d (f d \tau ) = * d V, 
\end{align}
in $U$.  That is, the 1-form $ i f d \tau$
is a local connection form in $U$.
Next for each $\alpha$ with $1\le \alpha \le n+1$, the 1-form 
\begin{align}
\omega_{\alpha} &= \i ( f + n + 1 - \alpha) d \tau, \label{conn001}
\end{align}
is well-defined on $U_\alpha$. 
Together, these $1$-forms define a global connection form (with values
in $\mathfrak{u}(1) = \i \RR$) on the total space $X_0 \rightarrow M$. 
That is, there is a global connection $\omega$ on $X_0$,
such that $\omega = \omega_{\alpha} +   \i\cdot d \theta$
over $U_{\alpha}$, 
where $\theta$ is an angular coordinate on the fiber.
\end{theorem}

\begin{proof} This was proved in \cite[Theorem 3.1]{HondaViaclovsky} for the 
case of two monopole points. It is straightforward to generalize the argument 
to the case of $n$ monopole points, so we only provide a brief sketch here. 
The Green's function is given by
\begin{align}
\Gamma_{(0,0,c)} (x,y,z) 
& =  - \frac{1}{2} + \frac{1}{2}
\left[
1 - \frac{ 4 c^2 z^2}{\left( r^2 + z^2 + c^2 \right)^2}
\right]^{-1/2},
\end{align}
where $r^2 = x^2 + y^2$, see \cite[Section 2]{LeBrun1991b}.
In the proof of \cite[Theorem 3.1]{HondaViaclovsky},
it is shown that
\begin{align} 
d ( f_{c_{\alpha}} d \tau) = *d ( \Gamma_{p_{\alpha}}).
\end{align}
It is easy to see that
\begin{align}
f_{c_{\alpha}}(0,z) = 
\begin{cases}
-1 &  z < c_{\alpha}\\
0 &  z > c_{\alpha}.\\
\end{cases}
\end{align}
From these it follows that the sum $f : = f_{c_1} + f_{c_2} + \cdots + f_{c_n}$ 
then satisfies 
\begin{align}
d (f d \tau ) = * d V, 
\end{align}
and
\begin{align}
f(0,z) = \alpha - n - 1, \hspace{3mm} z \in I_{\alpha}.
\end{align}
Consequently, the form 
\begin{align}
\omega_{\alpha} = \i( f + n + 1 - \alpha) d \tau,
\end{align}
extends smoothly to $U_{\alpha}$. It then follows from basic connection theory 
that the $\omega_{\alpha}$ are the local representatives of a globally 
defined connection. 
\end{proof}
Although we do not require this in the proof of our main theorem, 
we remark that one can use the above local connection forms to write down the 
transitions functions of the $\U(1)$-bundle explicitly:
\begin{proposition}
\label{trans}
With respect to the open covering \eqref{covering}, the transition 
functions of the $\U(1)$-bundle $\pi:X_0\to \mathcal H^3 \setminus \{p_1, p_2, \cdots, p_{n}\}$ are given by $g_{\alpha\beta} = e^{ \i (\beta - \alpha)\tau}$.
\end{proposition}
\begin{proof}
From above, we have that
\begin{align}
\omega_{\beta} - \omega_{\alpha} 
= \i ( f + n + 1 - \beta) d \tau - \i (f + 1 + n - \alpha) d \tau
= \i ( \alpha - \beta) d \tau. 
\end{align}
The formula for the change of connection is given by 
\begin{align}
\omega_\beta - \omega_\alpha = g_{\beta \alpha}^{-1} d g_{\beta \alpha},
\end{align}
which implies that $g_{\beta \alpha} = e^{\i ( \alpha - \beta)\tau}$,
or equivalently,  $g_{\alpha \beta} = e^{\i ( \beta - \alpha)\tau}$, 
\end{proof}

\section{Explicit identification with Joyce metrics}

In this section, we use the explicit connection forms from Section~\ref{explicit}
to prove Theorem~\ref{L=J}. 
As in Section~\ref{explicit}, $(r,\tau,z)$ denotes cylindrical coordinates
on $U = \H^3 \setminus \{z\mbox{-axis}\}$.
We introduce another coordinate system $(x_1,x_2)$ by setting
\begin{align}\label{coordchange2}
x_1 = r^2 - z^2,\quad
x_2 = 2 r z.
\end{align}
The map $(r,z)\mapsto (x_1,x_2)$ is a diffeomorphism from the quarter plane 
$\{(r,z)\set r>0,\,z>0\}$ to the upper half plane $\{(x_1,x_2)\set x_2>0\}$.
(Thus we adapt the upper-half plane model, rather than
 the right-half
plane model used in \cite{Joyce1995}.)
The point $(r,z)=(0,c_{\alpha})$ (on the boundary of $\{r>0,\,z>0\}$) 
determined from the monopole point $p_{\alpha}$, is mapped to the point 
$(x_1,x_2)=(-c_{\alpha}^2,0)$ (on the boundary of $\{x_2>0\}$).
In order to save space, for each integer $\alpha$ with
 $3\le \alpha\le n+2$, we put
\begin{align}
q_{\alpha}:=-c_{{\alpha}-2}^2,\quad
r_{\alpha}:=\sqrt{(x_1-q_{\alpha})^2+x_2^2},\quad
R:=\sqrt{x_1^2+x_2^2}
\end{align}
(we adopt this un-natural numbering for a later purpose).
Then we have $0>q_3>q_4>\cdots>q_{n+2}$, and also
\begin{align}
r^2 = \frac{1}{2} \left( R + x_1 \right),\quad
z^2 =  \frac{1}{2} \left( R - x_1 \right).
\end{align}
and
\begin{align}
dx_1^2 + dx_2^2 = 4(r^2 + z^2) ( dr^2 + dz^2) = 4 R \, ( dr^2 + dz^2).
\end{align}

Under the coordinates $(r,z,\tau,\theta)$, the  
metric $g_{\LB}$ multiplied by a conformal factor 
$(z^2V)^{-1}$
can be written as
\begin{align}
\begin{split}
\label{LBmetric2}
\frac{g_{\LB}}{z^2 V} &=  \frac{dr^2 + r^2 d \tau^2 + dz^2}{z^2}  + \frac 1{V^2} \, (d\theta+fd\tau)^2\\
&= \frac{ dr^2 + dz^2}{z^2}  + \frac{r^2}{z^2} d \tau^2
+ \frac1{V^2}\, ( d \theta + f d \tau )^2,
\end{split}
\end{align}
and noting $q_{\alpha}<0$ the functions $V$ and $f$ can be computed, in terms of the coordinates $(x_1,x_2)$, as
\begin{align}
V =
1-\frac n2+
\sum_{{\alpha}=3}^{n+2}\frac{R-q_{\alpha}}{2r_{\alpha}}, \quad
f(x_1,x_2) = - \frac n2 + \sum_{{\alpha}=3}^{n+2}
\frac{R+q_{\alpha}}{2r_{\alpha}}.
\label{f_and_V}
\end{align}
Hence, writing $g_{\mathcal H^2}:=(dx_1^2+dx_2^2)/x_2^2$, we have
\begin{align}
\label{LB1}
\frac{g_{\LB}}{z^2 V} 
&=
\frac{
\frac{dx_1^2+dx_2^2}{4R}
}{\frac{1}{2} \left( R  - x_1\right)}+
\frac
{ R  + x_1 }
{ R  - x_1 }
d\tau^2
+
\frac{( d \theta + f d \tau )^2 }{V^2}
\\
&=
\frac{x_2^2}{2R(R-x_1)}
\left[
g_{\mathcal H^2}  + 
\frac{2R^2}{x_2^2}
\left\{  \left( 1  + \frac{x_1}{ R } \right) d \tau^2
+  \left( 1  - \frac{x_1}{ R } \right)  
 \frac{( d \theta + f d \tau ) ^2 } {V ^2 } \right\}
\right].\label{LBmetric3}
\end{align}
From \eqref{f_and_V}, this expresses a toric LeBrun metric 
in terms of the coordinates $(x_1,x_2,\tau,\theta)$.
In the following, for simplicity of notation, we denote by 
$\tilde g_{\LB}$ the quantity
in the brackets $[\quad]$ in \eqref{LBmetric3}; namely we define
\begin{align}\label{tildeLBmetric}
\tilde g_{\LB}
:=\frac{2R(R-x_1)}{x_2^2z^2V} \, g_{\LB}.
\end{align}

Next we explain the explicit form of Joyce metrics  
on $n\#\mathbb{CP}^2$ of arbitrary type, following \cite{Joyce1995}. 
Let $k:=n+2$, and $\mathfrak q_1>\mathfrak q_2>\cdots>\mathfrak q_{k}$ 
be the set of elements in $\mathbb R\cup\{\infty\}$
 involved in the construction of Joyce metrics
(\cite[Theorem 3.3.1]{Joyce1995},
where the letter $p_i$ was used instead of $\mathfrak q_{\alpha}$).
For each ${\alpha}$ with $1\le {\alpha}\le k$ and $\mathfrak q_{\alpha}\neq 0,\infty$, 
let $\rho_{\alpha}:=\{(x_1-\mathfrak q_{\alpha})^2+x_2^2\}^{1/2}$, and let $u^{(\mathfrak q_{\alpha})}$ be an $\mathbb R^2$-valued function  defined by
\begin{align}
u^{(\mathfrak q_{\alpha})}(x_1,x_2)
=
\begin{pmatrix}
u_1^{(\mathfrak q_{\alpha})}(x_1,x_2) \\
u_2^{(\mathfrak q_{\alpha})}(x_1,x_2)
\end{pmatrix}
\quad\text{where}\quad
u_1^{(\mathfrak q_{\alpha})} = \frac{x_2}{\rho_{\alpha}},
\quad
u_2^{(\mathfrak q_{\alpha})} = \frac{x_1-\mathfrak q_{\alpha}}{\rho_{\alpha}}
\end{align}
(in \cite{Joyce1995} the notation $f^{(p_i)}$ is used instead of 
$u^{(\mathfrak q_{\alpha})}$).
When $\mathfrak q_{\alpha}=\infty$ or $\mathfrak q_{\alpha}=0$, we let
\begin{align}
u^{(\infty)}=
\begin{pmatrix}
0\\-1
\end{pmatrix},
\quad
u^{(0)}=
\begin{pmatrix}
x_2/R\\x_1/R
\end{pmatrix}.
\end{align}
Let $\{(m_{\alpha},n_{\alpha})\set 1\le {\alpha}\le k\}$ be the set of  pairs of coprime integers determined from the $\U(1) \times \U(1)$-action on $n\#\CP^2$ we are considering (namely, the stabilizer data).
Without loss of generality, we can always suppose that $m_{\alpha}n_{\alpha+1}-m_{\alpha+1}n_{\alpha}=-1$ for ${\alpha}$ with $1\le {\alpha}<k$, $(m_1,n_1)=(0,1)$ and $(m_{k}, n_{k})=(1,0)$, and also $m_{\alpha}>0,n_{\alpha}>0$ for any $ {\alpha}$ with $1<{\alpha}<k$. After this normalization, 
 we let
\begin{align}
\phi = 
\sum_{{\alpha}=1}^{k-1}  \frac{u^{(\mathfrak q_{\alpha})}-u^{(\mathfrak q_{{\alpha}+1})}}{2}
\otimes (m_{\alpha},n_{\alpha})\,
+  \frac{u^{(\mathfrak q_{k})}+u^{(\mathfrak q_{1})}}{2}
\otimes (m_{k},n_{k}).
\end{align}
If we write this as
$$
\phi=
\begin{pmatrix}
a_1 (x_1, x_2) &  b_1 (x_1, x_2) \\
a_2 (x_1, x_2) &  b_2 (x_1, x_2)
\end{pmatrix},
$$
then on the dense open subset $\mathcal H^2\times \U(1) \times \U(1)$ of $n\#\CP^2$ the Joyce metric  with the given $\U(1) \times \U(1)$-action is expressed as
\begin{align}\label{Joyce1}
g_J=g_{\mathcal H^2} + 
\frac{\left(a_1^2+a_2^2
\right)dy_1^2
+
\left(b_1^2+b_2^2
\right)dy_2^2
-2
\left(
a_1b_1+a_2b_2
\right)
dy_1 dy_2}{ (a_1b_2-a_2b_1 )^2 },
\end{align}
where $y_1,y_2$ are coordinates with period $2\pi$ on $\U(1) \times \U(1)$.

\begin{proposition}
If a Joyce metric has a  $\U(1)$-subgroup
of $\U(1)\times \U(1)$ acting semi-freely on $n\#\CP^2$,
then the stabilizer data can be supposed to be 
\begin{align}\label{stab0}
(m_1,n_1)=(0,1),\quad\text{and}\quad
(m_{\alpha},n_{\alpha})=(1,k-{\alpha})\quad\text{for}\quad2\le {\alpha}\le k.
\end{align}
\end{proposition}
\begin{proof}
For this, as in Proposition 3.1.1 of \cite{Joyce1995}, 
by choosing appropriate $\ZZ$-basis of $\mathbb Z^2$, we can always normalize the 
stabilizer data $\{(m_{\alpha},n_{\alpha})\}$ in a way that 
they satisfy 
\begin{align}\label{stab1}
(m_1,n_1)= (0,1), \, (m_k,n_k) = (1,0),\,\,
m_{\alpha}n_{\alpha+1} - m_{\alpha+1}n_{\alpha} = -1
{\text{ for }} 1\le \alpha< k.
\end{align}
The last condition in particular means that $(m_{\alpha},n_{\alpha})$ moves in the clockwise direction
as $\alpha$ increases.
Therefore $m_{\alpha}>0$ and $n_{\alpha}>0$ hold for $1<\alpha<k$.
As in \cite{Joyce1995} for mutually coprime integers $m$ and $n$ define 
a $\U(1)$-subgroup $G(m,n)$ of $\U(1) \times \U(1)$ by 
$$
G(m,n) = \{(e^{2\pi i \phi}, e^{2\pi i \psi} ) \set e^{2\pi i (m\phi+n\psi)} = 1\}.
$$
Then in $n\#\CP^2$ for each   $1\le \alpha\le n+2$
there exists a distinguished $\U(1)\times \U(1)$-invariant 2-sphere whose stabilizer is 
exactly $G(m_{\alpha}, n_{\alpha})$.
Let $S^2_{\alpha}$ be this 2-sphere.
(The union of all these spheres are exactly the complement of $\mathcal H^2 \times 
\U(1)\times \U(1)$ in $n\#\CP^2$.)
It is elementary to see from \eqref{stab1} that if $G(m,n)\subset \U(1)\times \U(1)$ is a $\U(1)$-subgroup 
which acts semi-freely on the first sphere $S^2_1$, then 
$(m,n) = (0,1)$ or otherwise $m=1$, up to simultaneous inversion of the sign.
Similarly, if $G(m,n)\subset \U(1)\times \U(1)$ is a $\U(1)$-subgroup 
which acts semi-freely on the last sphere $S^2_k$, then 
$(m,n) = (1,0)$ or otherwise $n=1$, up to simultaneous inversion of the sign.
Taking intersection of these,
any $\U(1)$-subgroup acting semi-freely on $n\#\CP^2$ has to be of 
the form $G(1,1), G(1,0)$ or $G(0,1)$.
But again it is elementary to see that
the subgroup $G(1,1)$ cannot act semi-freely on $S^2_{\alpha}$,
$1<\alpha<k$.
Therefore the two subgroups $G(1,0)$ and $G(0,1)$ are all subgroups
that can act semi-freely on $n\#\CP^2$.
If $G(1,0)$ (resp.\,$G(0,1)$) acts semi-freely on $S^2_{\alpha}$ ($1<\alpha<k$), it follows that 
$n_{\alpha} = 1$ (resp.\,$m_{\alpha}=1$).
Thus the stabilizer data must be
\begin{align}\label{stab2}
(m_1,n_1)= (0,1), \, (m_k,n_k) = (1,0),\,\,
(m_{\alpha},n_{\alpha}) = (m_{\alpha},1)
{\text{ for }} 1< \alpha< k,
\end{align}
for some $m_{\alpha}>0$, or 
\begin{align}\label{stab3}
(m_1,n_1)= (0,1), \, (m_k,n_k) = (1,0),\,\,
(m_{\alpha},n_{\alpha}) = (1,n_{\alpha})
{\text{ for }} 1< \alpha< k,
\end{align}
for some $n_{\alpha}>0$.
But of course these represent  the same $\U(1)\times U(1)$-action on 
$n\#\CP^2$, so we dispose of the former.
Then the final condition in \eqref{stab1} means $m_{\alpha} = k-\alpha$,
and we are done.
\end{proof}

Next, by the usual $\rm{PSL}(2,\mathbb R)$-action, we may suppose that
$\mathfrak q_{1}=\infty$ and $\mathfrak q_{2}=0$. 
From these normalizations, we compute
\begin{align}
 \phi &= \frac{u^{(\mathfrak q_1)}-u^{(\mathfrak q_2)}}{2} \otimes (0,1) + \sum _{{\alpha}=2} ^{k-1} \frac{u^{(\mathfrak q_{\alpha})}-u^{(\mathfrak q_{{\alpha}+1})}}{2} \otimes (1,n+2-{\alpha}) + \frac{u^{(\mathfrak q_{k})}+u^{(\mathfrak q_1)}}{2} \otimes (1,0)
\notag\\
&=
\frac12 \left( u^{(\mathfrak q_1)} + u^{(\mathfrak q_2)} \,,\,
u^{(\mathfrak q_1)} + (k-3) u^{(\mathfrak q_2)} - \sum_{{\alpha}=3}^k u^{(\mathfrak q_{\alpha})}  \right)
\notag\\
&= \frac12 
\begin{pmatrix}
\displaystyle\frac{x_2}R & (k-3)\displaystyle\frac{x_2}R-\displaystyle\sum_{{\alpha}=3}^k\frac{x_2}{\rho_{\alpha}}
\\ \displaystyle\frac{x_1}R-1 & \ \ \ (k-3)\displaystyle\frac{x_1}R-\displaystyle\sum_{{\alpha}=3}^k\frac{x_1-\mathfrak q_{\alpha}}{\rho_{\alpha}}-1
\end{pmatrix}
\left(=
\begin{pmatrix}a_1 & b_1 \\ a_2 & b_2
\end{pmatrix} \right).\label{JLB}
\end{align}
Substituting these into \eqref{Joyce1}, we obtain the explicit form of Joyce 
metrics which admit a semi-free $\U(1)$-action.

We next have the following
\begin{theorem}\label{thm:metrics}
With respect to the above coordinates and the identification $q_{\alpha}=\mathfrak q_{\alpha}$ so that $r_{\alpha}=\rho_{\alpha}$ $(3\le {\alpha}\le n+2)$,
the toric LeBrun metric $\tilde g_{\LB}$ (defined in \eqref{tildeLBmetric}) and the Joyce metric $g_J$ (defined in \eqref{Joyce1} with \eqref{JLB}) are isometric under the map
\begin{align}
(x_1,x_2,\theta,\tau)\longmapsto
(x_1,x_2,y_1,y_2)
=\left(x_1,x_2,\theta,\tau\right).
\end{align}
\end{theorem}

For the proof, we begin with the following

\begin{lemma}\label{lemma:proportional}
As functions on $\mathcal H^2=\{(x_1,x_2)\set x_2>0\}$,
we have the following relationship
\begin{align}\label{proport}
a_1b_2-a_2b_1=-\frac{x_2}{2R}V.
\end{align}
\end{lemma}

\begin{remark}{\em
The negativity of $a_1b_2-a_2b_1$ seemingly contradicts Lemma 3.3.3 
in \cite{Joyce1995}, but this is not a problem, since the sign 
comes from the difference of  the orientation on the right half plane 
used in Joyce's paper and that on the upper half plane used in this paper. }
\end{remark}
\proof
By \eqref{JLB}, we have
\begin{align*}
4 ( a_1 b_2 - a_2 b_1 ) & =  \frac {x_2} R
\left\{
(k-3)\displaystyle\frac{x_1}R-\displaystyle\sum_{{\alpha}=3}^k\frac{x_1-\mathfrak q_{\alpha}}{\rho_{\alpha}}-1\right\}
- \big(\frac{x_1} R-1\big)
\left\{
(k-3)\displaystyle\frac{x_2}R-\displaystyle\sum_{{\alpha}=3}^k\frac{x_2}{\rho_{\alpha}}
\right\},
\end{align*}
and after several cancellations, this equals
$$
\frac{x_2}{R}\left\{
k-4-\sum_{{\alpha}=3}^{k}\frac{R-\mathfrak q_{\alpha}}{\rho_{\alpha}}\right\}
$$
which is exactly $-(2x_2V)/R$, under our assumption $q_{\alpha}=\mathfrak q_{\alpha}$.
Dividing by $4$ gives the claim of the lemma.
\proofend

\vspace{2mm}\noindent
{\em Proof of Theorem \ref{thm:metrics}.}
We write the two metrics as 
$$
\tilde g_{\LB} = g_{\mathcal H^2} + \tilde g_{11} d\theta^2 + 2 \tilde g_{13} d\theta d\tau + \tilde g_{33} d\tau^2,
$$
and
$$
g_J = g_{\mathcal H^2} +  g_{11} dy_1^2 + 2  g_{12} dy_1 dy_2 +  g_{22} dy_2^2.
$$
Then the claim of Theorem \ref{thm:metrics} is equivalent to the three identities
\begin{align}
g_{11}=\tilde g_{11},\quad
g_{12}=\tilde g_{13},\quad
g_{22}=\tilde g_{33}.
\end{align}
In the following, for simplicity of notation, we write
$$
\sum_{{\alpha}=3}^{n+2}=:\sum{}'\quad {\text{and}}\quad
\sum_{3\le \alpha<\beta\le n+2}=:\sum{}''
$$

First, we readily have
\begin{align}
\tilde g_{11} = \frac{2R^2}{x_2^2}
\frac{1-\frac{x_1}R}{V^2},
\end{align}
and also by using Lemma \ref{lemma:proportional}
\begin{align}
g_{11} &= \frac{a_1^2+a_2^2}{(a_1b_2-a_2b_1)^2} 
 = \frac{ \frac12 \left( 1 - \frac{x_1}R \right)}
{(\frac{x_2}{2R}V)^2}
= \frac{2R^2}{x_2^2}
\frac{1-\frac{x_1}R}{V^2}.
\end{align}
Therefore we obtain $g_{11}=\tilde g_{11}$.

Second, from \eqref{LBmetric3}, we have
\begin{align}\label{LBmetric5}
\tilde g_{13} = \frac{2R^2}{x_2^2}
\left(
1-\frac{x_1}{R}
\right)
\frac{f}{V^2}
= \frac{2R^2}{x_2^2}
\left(
1-\frac{x_1}{R}
\right)
\frac{\displaystyle -\frac n2 + \frac 12 \sum{}'
\frac{R+q_{\alpha}}{r_{\alpha}}}{V^2}.
\end{align}
On the other hand, from \eqref{JLB} we can compute,
by using the relation $x_1^2+x_2^2=R^2$ twice, 
\begin{multline}
4 ( a_1b_1 + a_2b_2 ) \\=
\frac{x_2}R
\left\{
(n-1)\frac{x_2}R-\sum{}'\frac{x_2}{\rho_{\alpha}}
\right\}
+  \left(\frac{x_1}{R}-1\right)
\left\{
(n-1)\frac{x_1}R-\sum{}'
\frac{x_1-\mathfrak q_{\alpha}}{\rho_{\alpha}}
-1\right\}\\
=
\left(  n - \sum{}' \frac{R+\mathfrak q_{\alpha}}{\rho_{\alpha}} \right)
\left( 1- \frac{x_1}R \right).
\end{multline}
Hence again by using Lemma \ref{lemma:proportional} we obtain
\begin{align}\label{Jmetric2}
g_{12} = 
-\frac{a_1b_1+a_2b_2}{(a_1b_2-a_2b_1)^2}
=
-
\frac{\frac14\left(  n - \sum{}' \frac{R+\mathfrak q_{\alpha}}{\rho_{\alpha}} \right)
\left( 1- \frac{x_1}R \right)}
{\displaystyle \left(\frac{x_2}{2R}V \right)^2}.
\end{align}
By comparing \eqref{LBmetric5} and \eqref{Jmetric2}, we obtain $g_{12}=\tilde g_{13}$.

Finally, for the remaining coefficients $\tilde g_{33}$ and $g_{22}$, we have, by \eqref{LBmetric3}, 
\begin{align}
\tilde g_{33} &= 
\frac{2R^2}{x_2^2}
\left\{
\left(
1 + \frac{x_1}{R} 
\right)
+
\left(
1 - \frac{x_1}{R} 
\right)
\frac{f^2}{V^2}
\right\} =
\frac{2R^2}{x_2^2V^2}
\left\{
(V^2+f^2) + \frac{x_1}R (V^2-f^2) 
\right\}.
\end{align}
Further by \eqref{f_and_V} we compute 
\begin{multline}
V^2+f^2 = 
1 - n + \frac{n^2}2 \\
+ \sum{}'
\frac{R-q_{\alpha}}{r_{\alpha}} - n \sum{}' \frac{R}{r_{\alpha}}
+ \frac12
\left(
\sum{}' \frac{R^2 + q_{\alpha}^2}{r_{\alpha}^2}
+ 2 \sum{}'' \frac{R^2 + q_{\alpha}q_{\beta}}{r_{\alpha}r_{\beta}}
\right),
\end{multline}
and
\begin{align}
V^2 - f^2 &= 
1 - n  + \sum{}'
\frac{R-q_{\alpha}}{r_{\alpha}} + n \sum{}' \frac{q_{\alpha}}{r_{\alpha}}
- R
\left(
\sum{}' \frac{q_{\alpha}}{r_{\alpha}^2}
+  \sum{}'' \frac{ q_{\alpha} + q_{\beta} }{r_{\alpha} r_{\beta}}
\right).
\end{align}
From these we obtain
\begin{multline}
\tilde g_{33} = \frac{2R^2}{x_2^2V^2}
\Big\{ \frac12\sum{}' \frac{R^2+q_{\alpha}^2-2x_1q_{\alpha}}{r_{\alpha}^2} + \sum{}'' \frac{R^2+q_{\alpha}q_{\beta}-x_1(q_{\alpha}+q_{\beta})}{r_{\alpha}r_{\beta}}\\
+ \sum{}' \frac{(1-n)R-q_{\alpha}}{r_{\alpha}}
+ \frac{x_1}R \sum{}' \frac{R+(n-1)q_{\alpha}}{r_{\alpha}}
+ (1-n)\frac{x_1}R + 1 - n + \frac{n^2}2\Big\}.
\end{multline}
Noting the relation $R^2+q_{\alpha}^2-2x_1q_{\alpha}=r_{\alpha}^2$,
the first summation becomes just $n$, so this equals
\begin{multline}\label{g33}
 \frac{2R^2}{x_2^2V^2}
\Big\{ \sum{}'' \frac{R^2+q_{\alpha}q_{\beta}-x_1(q_{\alpha}+q_{\beta})}{r_{\alpha} r_{\beta}}
+ \sum{}' \frac{(1-n)R-q_{\alpha}}{r_{\alpha}}
+ \frac{x_1}R \sum{}' \frac{R+(n-1)q_{\alpha}}{r_{\alpha}}\\
+ (1-n)\frac{x_1}R + 1 - \frac n2 + \frac{n^2}2
\Big\}.
\end{multline}
On the other hand, by \eqref{JLB}, we can compute
\begin{multline}
4(b_1^2+b_2^2)=(n-1)^2-2(n-1)
\left(
\frac{x^2_2} R \sum{}'  \frac{1}{\rho_{\alpha}}
+  \frac{x_1} R \sum{}' \frac{x_1-\mathfrak q_{\alpha}}{\rho_{\alpha}}
+ \frac {x_1}{R}
\right)  \\
+  x_2^2\left(\sum{}'\frac1{\rho_{\alpha}}\right)^2
+  \left(\sum{}'\frac{x_1-\mathfrak q_{\alpha}}{\rho_{\alpha}}\right)^2
+  2\sum{}' \frac{x_1-\mathfrak q_{\alpha}}{\rho_{\alpha}} +1 \\
= (n-1)^2 - 2(n-1) \left\{
\frac1R \sum{}' \frac{R^2-x_1\mathfrak q_{\alpha}}{\rho_{\alpha}} + \frac{x_1}R
\right\}
+ \sum{}' \frac{x_2^2+(x_1-\mathfrak q_{\alpha})^2}{\rho_{\alpha}^2} \\
+ 2 \sum{}'' \frac{x_2^2 + (x_1-\mathfrak q_{\alpha})(x_1 - \mathfrak q_{\beta})}{\rho_{\alpha} \rho_{\beta}}
+ 2 \sum{}' \frac{x_1-\mathfrak q_{\alpha}}{\rho_{\alpha}} + 1\\
= 2 \sum{}'' \frac{R^2+\mathfrak q_{\alpha}\mathfrak q_{\beta}-x_1(\mathfrak q_{\alpha}+\mathfrak q_{\beta})}{\rho_{\alpha} \rho_{\beta}}
+ \frac{2x_1}R \sum{}' \frac{R+(n-1)\mathfrak q_{\alpha}}{\rho_{\alpha}}\\
+ 2 \sum{}' \frac{(1-n)R-\mathfrak q_{\alpha}}{\rho_{\alpha}}
+ 2(1-n)\frac{x_1}R + n^2 - n + 2.
\end{multline}
Under the identification $q_{\alpha}=\mathfrak q_{\alpha}$,
 this is exactly twice the quantity in the braces $\{\quad\}$ in \eqref{g33}.
Consequently,
\begin{align}\label{aaaa}
\tilde g_{33} = \frac{2R^2}{x_2^2V^2}\cdot 2 (b_1^2+b_2^2)
= \frac{4R^2}{x_2^2V^2}(b_1^2+b_2^2).
\end{align}
On the other hand by \eqref{Joyce1} 
and Lemma \ref{lemma:proportional} we have
$$g_{22} = \frac{b_1^2+b_2^2}{(a_1b_2-a_2b_1)^2}
= \frac{b_1^2+b_2^2}{
\left(\frac{x_2}{2R}V \right)^2
}.$$
Hence from \eqref{aaaa} we obtain 
\begin{align}
\tilde g_{33} = \frac{4R^2}{x_2^2V^2} 
\left(\frac{x_2}{2R}V \right)^2 g_{22} ={g_{22}},
\end{align}
as required, which 
completes the proof of Theorem \ref{thm:metrics}.
\proofend


\vspace{10mm}

\begin{thebibliography}{LeB91b}

\bibitem[Fuj00]{Fujiki2000}
Akira Fujiki.
\newblock Compact self-dual manifolds with torus actions.
\newblock {\em J. Differential Geom.}, 55(2):229--324, 2000.

\bibitem[GH78]{GibbonsHawking}
G.~W. Gibbons and S.~W. Hawking.
\newblock Gravitational multi-instantons.
\newblock {\em Physics Letters B}, 78(4):430--432, 1978.

\bibitem[HV09]{HondaViaclovsky}
Nobuhiro Honda and Jeff Viaclovsky.
\newblock Conformal symmetries of self-dual hyperbolic monopole metrics.
\newblock arXiv.org:0902.2019, 2009, to appear in {\em{Osaka J. Math.}}

\bibitem[Joy95]{Joyce1995}
Dominic~D. Joyce.
\newblock Explicit construction of self-dual {$4$}-manifolds.
\newblock {\em Duke Math. J.}, 77(3):519--552, 1995.

\bibitem[LeB91a]{LeBrun1991b}
Claude LeBrun.
\newblock Anti-self-dual {H}ermitian metrics on blown-up {H}opf surfaces.
\newblock {\em Math. Ann.}, 289(3):383--392, 1991.

\bibitem[LeB91b]{LeBrun1991}
Claude LeBrun.
\newblock Explicit self-dual metrics on {${\bf C}{\rm P}\sb 2\#\cdots\#{\bf
  C}{\rm P}\sb 2$}.
\newblock {\em J. Differential Geom.}, 34(1):223--253, 1991.

\bibitem[LeB93]{LeBrun1993}
Claude LeBrun.
\newblock Self-dual manifolds and hyperbolic geometry.
\newblock In {\em Einstein metrics and {Y}ang-{M}ills connections ({S}anda,
  1990)}, volume 145 of {\em Lecture Notes in Pure and Appl. Math.}, pages
  99--131. Dekker, New York, 1993.

\bibitem[Po86]{Poon1986}
Y.~Sun Poon.
\newblock Compact self-dual manifolds with positive scalar curvature.
\newblock {\em J. Differential Geom.}, 24(1):97--132, 1986.

\end{thebibliography}

\end{document}